\documentclass[a4paper,12pt]{article}

\usepackage{graphicx}
\usepackage[utf8]{inputenc}
\usepackage[english]{babel}
\usepackage{amssymb}
\usepackage{makeidx}
\usepackage{amsmath}
\usepackage{amsthm}
\usepackage{authblk}
\usepackage{pxfonts}
\usepackage{tabu}
\usepackage{multicol}
\usepackage{enumitem}
\usepackage{appendix}

\allowdisplaybreaks[4]

\newtheorem{teo}{Theorem}[section]
\newtheorem{lemma}{Lemma}[section]

\newtheorem{oss}{Remark}[section]
\newtheorem{prop}{Proposition}[section]

\newcommand{\ra}{\longrightarrow}
\newcommand{\dis}{\displaystyle}
\newcommand{\bs}{\backslash}
\newcommand{\ov}{\overline}
\newcommand{\R}{\mathbb{R}}

\newcommand{\C}{\mathbb{C}}
\newcommand{\mb}{\mathbb}

\newcommand{\eps}{\varepsilon}

\newcommand{\ph}{\varphi}
\newcommand{\tabitem}{\hspace{0.3cm}~~\llap{\textbullet}~~}
\newcommand{\D}[2]{ \frac{\partial #1}{ \partial #2}}
\newenvironment{Si}[1]{\left\{\begin{array}{#1}}{\end{array} \right. }
\DeclareMathOperator{\dv}{div}

\makeindex
\title{Onofri-type inequalities for singular Liouville equations}
\author{Gabriele Mancini\thanks{S.I.S.S.A/I.S.A.S, Via Bonomea 265, 34136 Trieste (Italy) - gmancini@sissa.it\\The author is supported by the FIRB project {\em Analysis and Beyond}, by the PRINs {\em Variational Methods and Nonlinear PDE's} and {\em Variational and perturbative aspects of nonlinear differential problems} and by the Mathematics Department at the University of Warwick.}}
\date{}

\begin{document}
\maketitle
\begin{abstract} 
We study the blow-up behaviour of minimizing sequences for the singular Moser-Trudinger functional on compact surfaces. Assuming non-existence of minimum points, we give an estimate for the infimum value of the functional. This result can be applied to give  sharp Onofri-type inequalities on the sphere in the presence of at most two singularities. 
\end{abstract}

\section{Introduction}
Let $(\Sigma,g)$ be a smooth, compact Riemannian surface; the standard Moser-Trudinger inequality (see \cite{mos}, \cite{fontana}) states that 
\begin{equation}\label{MT}
\log\left(\frac{1}{|\Sigma|}\int_{\Sigma} e^{u-\ov{u}}dv_g\right)\le \frac{1}{16\pi}\int_{\Sigma} |\nabla_g u|^2 dv_g + C(\Sigma,g)\qquad \forall\; u\in H^1(\Sigma)
\end{equation}
where $C(\Sigma,g)$ is a constant depending only on $\Sigma$ and $g$, and the coefficient $\frac{1}{16\pi}$ is optimal. A sharp version of (\ref{MT}) was proved by Onofri in \cite{onofri} for the sphere endowed with the standard Euclidian metric $g_0$. He identified the sharp value of $C$ and the family of functions realizing equality, proving 
\begin{equation}\label{onofri}
\log\left(\frac{1}{4\pi}\int_{S^2} e^{u-\ov{u}} dv_{g_0}\right) \le \frac{1}{16\pi}\int_{S^2} |\nabla_{g_0} u|^2dv_{g_0}
\end{equation}
with equality holding if and only if the metric $e^{u}g$ has constant positive Gaussian curvature, or, equivalently, $u=\log|\det d \ph|+c$ with 
$c\in \R$ and $\ph$ conformal diffeomorphism of $S^2$. Onofri's inequality played an important role (see \cite{CY1}, \cite{CY2}) in the variational approach to the equation
$$
\Delta_{g_0} u +K  \;e^{u}=1
$$
which is connected to the classical problem of prescribing the Gaussian curvature of $S^2$.
In this paper we will consider extensions of Onofri's result in connection with the study of the more general equation
\begin{equation}\label{eq sing}
-\Delta_g v = \rho  \left( \frac{K e^{v}}{\int_{\Sigma} K e^{v}dv_g}-\frac{1}{|\Sigma|} \right)-4\pi \sum_{i=1}^m \alpha_i\left(\delta_{p_i}-\frac{1}{|\Sigma|}\right),
\end{equation}
where $K\in C^\infty(\Sigma)$ is a positive function, $\rho>0$, $p_1,\ldots, p_m\in \Sigma$ and $\alpha_1,\ldots,\alpha_m \in (-1,+\infty)$.
This is known as the singular Liouville equation and arises in several problems in Riemannian geometry and mathematical physics. When $(\Sigma,g)=(S^2,g_0)$ and $\rho = 8\pi +4\pi \sum_{i=1}^m \alpha_i$, solutions of (\ref{eq sing}) provide  metrics on $S^2$ with prescribed Gaussian curvature $K$ and conical singularities of angle $2\pi(1+\alpha_i)$ (or of order $\alpha_i$) in $p_i$, $i=1,\ldots,m$ (see for example \cite{BarMalDeM}, \cite{ChenLi}, \cite{troyanov}). Equation (\ref{eq sing}) also appears in the description of Abelian Chern-Simons vortices in superconductivity and Electroweak theory (\cite{HongKim}, \cite{Tarantello}). We refer to \cite{BarMal}, \cite{Carl}, \cite{CarMal1}, \cite{CarMal2}, \cite{MalRui}, for some recent existence results.
Systems of Liouville equations have applications in the description of holomorphic curves in $\C\mathbb{P}^n$ (\cite{BolWoo}, \cite{Calabi}) and in the nonabelian Chern-Simons theory which might have applications in high temperature superconductivity (see \cite{Tar2} and references therein). 
Denoting by $G_p$ the Green's function at $p$, namely the solution of 
\begin{equation*}
\begin{Si}{l}
-\Delta_g G_{p} = \delta_p- \frac{1}{|\Sigma|}\\
\int_{\Sigma} G_p\; dv_g =0
\end{Si},
\end{equation*}
the change of variables
$$
u= v+ 4\pi\sum_{i=1}^m \alpha_i G_{p_i} 
$$
transforms \eqref{eq sing} into
\begin{equation}\label{eq nonsing}
-\Delta_g u = \rho  \left( \frac{h e^{u}}{\int_{\Sigma} h e^{u}dv_g}-\frac{1}{|\Sigma|} \right)
\end{equation}
where 
\begin{equation}\label{h}
h= K\prod_{1\le i\le m} e^{-4\pi \alpha_i G_{p_i}}
\end{equation}
satisfies
\begin{equation}\label{dist}
h(p)\approx d(p,p_i)^{2\alpha_i} \mbox{ for } p\approx p_i.
\end{equation}

In \cite{troyanov}, studying curvature functions for surfaces with conical singularities, Troyanov proved that if $h\in C^\infty(\Sigma\bs\{p_1,\ldots,p_m\})$ is a positive function satisfying  \eqref{dist}, then
\begin{equation}\label{troy}
\log\left(\frac{1}{|\Sigma|}\int_{\Sigma} h \; e^{u-\ov{u}}dv_g\right) \le \frac{1}{16\pi \dis{ \min\left\{1,1+\min_{1\le i\le m}\alpha _i\right\}} }\int_{\Sigma} |\nabla_g u|^2 dv_g+ C(\Sigma,g,h).
\end{equation}

The optimal  constant  $C(\Sigma,g,h)$ can be obtained by minimizing the functional
$$
J_{\ov{\rho}}(u)=\frac{1}{2}\int_{\Sigma} |\nabla_g u|^2 dv_g+ \frac{\ov{\rho}}{|\Sigma|} \int_{\Sigma} u\; dv_g -\ov{\rho} \log\left(\frac{1}{|\Sigma|}\int_{\Sigma} he^{u}dv_g\right),
$$
where $\ov{\rho} =\dis{\min\left\{1,1+\min_{1\le i\le m}\alpha _i\right\} } $. 
In this paper we will assume non-existence of minimum points for $J_{\ov{\rho}}$ and exploit known blow-up results (\cite{bar-tar}, \cite{BarLinTar}, \cite{bar-mont})  to describe the behaviour of a suitable minimizing sequence and compute $\dis{\inf_{H^1(\Sigma)}J_{\ov{\rho}}}$. The same technique was used by Ding, Jost, Li and Wang  \cite{djlw} to give an existence result for (\ref{eq sing}) in the regular case. From their proof it follows that if $\alpha_i=0$  $\forall\; i$ and there is no minimum point for $J_{\ov{\rho}}$, then
$$
 \inf_{H^1(\Sigma)} J_{\ov{\rho}} =  -8\pi\left( 1+\log\left(\frac{\pi}{|\Sigma|}\right)  +\max_{p\in \Sigma} \left\{4\pi A(p)+ \log h(p)\right\}\right)
$$ 
where $A(p)$ is the value in $p$ of the regular part of $G_p$. 
Here we extend this result to the general case proving:

\begin{teo}\label{teo 1}
Assume that $h$ satisfies \eqref{h} with $K\in C^\infty(\Sigma)$, $K>0$, $\alpha_i \in (-1,+\infty)\bs\{0\}$, and that there is no minimum point of $J_{\ov{\rho}}$. If $\alpha:=\dis{\min_{1\le i\le m} \alpha_i<0}$, then
$$
\inf_{H^1(\Sigma)} J_{\ov{\rho}}= -8\pi(1+\alpha) \left( 1+\log\left(\frac{\pi}{|\Sigma|}\right)+\max_{1\le i\le m,\alpha_i=\alpha}\left\{4\pi A(p_i)+ \log  \left(\frac{K(p_i)}{1+\alpha}\prod_{j\neq i} e^{-4\pi\alpha_j G_{p_j}(p_i)}\right)\right\}\right)
$$
while if $\dis{\alpha>0}$
$$
\inf_{H^1(\Sigma)} J_{\ov{\rho}}= -8\pi \left( 1+\log\left(\frac{\pi}{|\Sigma|}\right)+\max_{p\in \Sigma\bs \{p_1,\ldots,p_m\}} \left\{4\pi A(p)+ \log h(p)\right\}\right).
$$
\end{teo}

In the last part of the paper we consider the case of the standard sphere with  $K\equiv 1$ and at most two singularities. 
When $m=1$ a simple Kazdan-Warner type identity proves non-existence of solutions for $(\ref{eq nonsing})$.  Thus, one can apply theorem \ref{teo 1} to obtain the following sharp version of (\ref{troy}):

\begin{teo}\label{una sing}
If $h=e^{-4\pi \alpha_1 G_{p_1}}$ with $\alpha_1\neq 0$, then $\forall\; u\in H^1(S^2)$
$$\log\left(\frac{1}{4\pi}\int_{S^2} h e^{u-\ov{u}}dv_{g_0}\right)<\frac{1}{16\pi\min\{1,1+\alpha_1\}}\int_{S^2} |\nabla u|^2dv_{g_0}+\max\left\{\alpha_1,-\log(1+\alpha_1)\right\}.$$
\end{teo}

The same non-existence argument works for $m=2$, $\min\{\alpha_1,\alpha_2\}<0$ and $\alpha_1 \neq \alpha_2$ if the singularities are located in two antipodal points.

\begin{teo}\label{due sing}
Assume $h=e^{-4\pi \alpha_1 G_{p_1}-4\pi \alpha_2 G_{p_2}}$ with $p_2= -p_1$, $\alpha_1= \min\{\alpha_1,\alpha_2\}<0$ and $\alpha_1\neq \alpha_2$; then $\forall\; u\in H^1(S^2)$
$$\log\left(\frac{1}{4\pi}\int_{S^2} h e^{u-\ov{u}}dv_{g_0}\right)< \frac{1}{16\pi (1+\alpha_1)}\int_{S^2} |\nabla u|^2dv_{g_0}+ \alpha_2-\log(1+\alpha_1).$$
\end{teo}

When $\alpha_1=\alpha_2<0$ theorem \ref{teo 1}  cannot be directly applied because (\ref{eq nonsing}) has solutions. However, it is possible to use a stereographic projection and a classification result in \cite{praj-tar} to find an explicit expression for the solutions. In particular  a direct computation allows to prove that all the solutions are minimum points of $J_{\ov{\rho}}$ and to find the value of $\dis{\min_{H^1(S^2)}J_{\ov{\rho}}}$. 

\begin{teo}\label{due sing uguali}
Assume $h=e^{-4\pi \alpha \left( G_{p_1}+G_{p_2}\right)}$ with $\alpha<0$ and $p_1=-p_2$; then $\forall\; u\in H^1(S^2)$ we have 
$$\log\left(\frac{1}{4\pi}\int_{S^2} h e^{u-\ov{u}}dv_{g_0}\right)\le \frac{1}{16\pi(1+\alpha)}\int_{S^2} |\nabla u|^2dv_{g_0}+ \alpha-\log(1+\alpha).$$
Moreover the following conditions are equivalent: 
\begin{itemize}
\item $u$ realises equality.
\item If $\pi$ denotes the stereographic projection from $p_1$ then
$$
u\circ \pi^{-1}(y)= 2\log\left( \frac{(1+|y|^2)^{1+\alpha}}{1+e^\lambda|y|^{2(1+\alpha)}} \right)+c
$$
for some $\lambda,c \in \R$.
\item $h e^{u} g_0$ is a metric with constant positive Gaussian curvature and conical singularities of order $\alpha_i$ in $p_i$, $i=1,2$.
\end{itemize} 
\end{teo}

This is a generalization of Onofri's inequality \eqref{onofri} for metrics with two conical singularities.

\section{Preliminaries and blow up analysis}
Let $(\Sigma,g)$ be a smooth compact, connected, Riemannian surface and let $S:=\{p_1,\ldots,p_m\}$ be a finite subset of $\Sigma$. Let us consider a function $h$ satisfying \eqref{h} with $K\in C^\infty(\Sigma)$, $K>0$ and $\alpha_i\in (-1,+\infty)\bs\{0\}$. In order to distinguish the singular points of $h$ from the regular ones, we introduce  a singularity index function
$$
\beta(p):=\begin{Si}{cc}
\alpha_i & \mbox{if } p=p_i\\
0 & \mbox{if } p\notin S
\end{Si}.
$$
We will denote $\dis{\alpha:= \min_{p\in \Sigma} \beta (p)= \min\left\{\min_{1\le i\le m}\alpha_i,0\right\}}$ the minimum singularity order. We shall consider the functional
\begin{equation}\label{funzionale}
J_\rho(u)=\frac{1}{2}\int_{\Sigma} |\nabla_g u|^2 dv_g + \frac{\rho}{|\Sigma|} \int_{\Sigma} u\; dv_g -\rho\log\left(\frac{1}{|\Sigma|}\int_{\Sigma} he^{u}dv_g\right).
\end{equation}
Our goal is to give a sharp version of (\ref{troy}) finding the explicit value of 
\begin{equation}\label{C e inf}
C(\Sigma,g,h) = -\frac{1}{8\pi(1+\alpha)} \inf_{u\in H^1(\Sigma)} J_{8\pi(1+\alpha)} (u).
\end{equation}
To simplify the notations we will denote  $\ov{\rho}:=8\pi(1+\alpha)$, $\rho_\eps= \ov{\rho}-\eps$, $J_\eps:= J_{\rho_\eps}$ and  $J:= J_{\ov{\rho}}$. From (\ref{troy}) it follows that $\forall\; \eps >0$ the functional $J_\eps$ is coercive and, by direct methods, it is possible to find a function $u_\eps \in H^1(\Sigma)$ satisfying
\begin{equation}\label{minim}
J_\eps(u_\eps)= \inf_{u\in H^1(\Sigma)}J_\eps(u)
\end{equation}
and
\begin{equation}\label{eq desing}
-\Delta_g u_\eps = \rho_\eps \left( \frac{h e^{u_\eps}}{\int_{\Sigma} h e^{u_\eps}dv_g} -\frac{1}{|\Sigma|} \right).
\end{equation}

Since  $J_{\eps}$ is  invariant under addition of constants $\forall\; \eps >0$, we may also assume
\begin{equation}\label{norm}
\int_{\Sigma} h\; e^{u_\eps} dv_g = 1.
\end{equation}

\begin{oss}\label{oss reg}
$u_\eps\in C^{0,\gamma}(\Sigma)\cap W^{1,s}(\Sigma)$ for some $\gamma\in(0,1)$ and $s>2$.
\end{oss}
\begin{proof}
It is easy to see that $h\in L^q(\Sigma)$ for some $q >1$ ( $q=+\infty$  if $\alpha=0$ and $q< -\frac{1}{\alpha}$ for $\alpha <0)$. Applying locally Remarks 2 and 5 in \cite{BM} one can show that $u_\eps\in L^{\infty}(\Sigma)$ so $-\Delta{u_\eps}\in L^q(\Sigma)$ and by standard elliptic estimates $u_\eps \in W^{2,q}(\Sigma)$. Since $q>1$ the conclusion follows by Sobolev's embedding theorems. 
\end{proof}

The behaviour of $u_\eps$ is described by the following concentration-compactness result:

\begin{prop}\label{prop conc-comp}
Let $u_n$ be a sequence of solutions of 
$$
-\Delta_g u_n = V_n e^{u_n} -\psi_n
$$
where $\|\psi_n\|_{L^s(\Sigma)}\le C$ for some $s>1$, and $$V_n= K_n \prod_{1\le i\le m} e^{-4\pi \alpha_iG_{p_i}}$$ with  $K_n\in  C^\infty(\Sigma)$, $0<a\le K_n\le b$ and $\alpha_i >-1$, $i=1,\ldots,m$. Then there exists a subsequence $u_{n_k}$ of $u_n$ such that one of the following holds:
\begin{itemize}
\item[i.] $u_{n_k}$ is uniformly bounded in $L^\infty(\Sigma)$;
\item[ii.] $u_{n_k}\longrightarrow -\infty$ uniformly on $\Sigma$;
\item[iii.] there exist a finite blow-up set $B =\{q_1,\ldots,q_l\}\subseteq \Sigma$ and a corresponding family of sequences  $\{q^j_k\}_{k\in \mathbb{N}}$, $j=1,\ldots l$ such that $q_k^j\stackrel{k\to \infty}{\longrightarrow} q_j$ and $u_{n_k}(q_k^j)\stackrel{k\to \infty}{\longrightarrow} +\infty$ $j=1,\ldots,l$. Moreover  $u_{n_k}\stackrel{k\to \infty}{\longrightarrow} -\infty$ uniformly on compact subsets of $\Sigma \bs B $ and $V_{n_k}e^{u_{n_k}} \rightharpoonup \sum_{j=1}^l \beta_j \delta_{q_j}$ weakly in the sense of measures where $\beta_j = 8\pi(1+\beta(q_j))$ for $j=1,\ldots,l$.
\end{itemize}
\end{prop}

A proof of proposition \ref{prop conc-comp} in the regular case can be found in \cite{Li} while the general case is a consequence of the results in \cite{bar-tar} and \cite{bar-mont}.  In our analysis we will also need the following  local version of proposition \ref{prop conc-comp} proved by Li and Shafrir (\cite{Li-sha}):

\begin{prop}\label{LiS}
Let $\Omega$ be an open domain in $\R^2$ and $v_n$ be a sequence satisfying $\|e^{v_n}\|_{L^1(\Omega)}\le C$ and 
$$
-\Delta v_n = V_n e^{v_n}  
$$
where $0\le V_n\in C_0(\ov{\Omega})$ and $V_n\longrightarrow V$ uniformly in $\ov{\Omega}$. If $v_n$ is not uniformly bounded from above on compact subset of $\Omega$, then $V_n e^{v_n} \rightharpoonup  \dis{8\pi \sum_{i=1}^lm_j \delta_{q_j}}$ as measures, with $q_j\in \Omega$  and $m_j\in \mb{N}^+$, $j=1,\ldots,l$.
\end{prop}

Applying proposition \ref{prop conc-comp} to $u_\eps$ under the additional condition (\ref{norm}) we obtain that either $u_\eps$ is uniformly bounded in $L^\infty(\Sigma)$ or its blows-up set contains a single point $p$ such that $\beta(p)= \alpha$. In the first case, one can use elliptic estimates to find uniform bounds on $u_\eps$ in $W^{2,q}(\Sigma)$, for some $q>1$; consequently, a subsequence of $u_\eps$ converges in $H^{1}(\Sigma)$ to a function $u\in H^1(\Sigma)$ that is a minimum point of $J$ and a solution of (\ref{eq nonsing}) for $\rho=\ov{\rho}$. We now focus on the second case, that is 
\begin{equation}\label{blowup}
\lambda_\eps:= \max_{\Sigma}u_\eps = u_\eps (p_\eps) \ra +\infty \quad \mbox{ and } \quad p_\eps \ra p \quad \mbox{with} \quad \beta(p)=\alpha.
\end{equation}

By proposition \ref{prop conc-comp} we also get:

\begin{lemma}\label{lemma riass}
If $u_\eps$ satisfies (\ref{eq desing}), (\ref{norm}) and (\ref{blowup}), then, up to subsequences,
\begin{enumerate}[ref=\arabic*.]
\item\label{l1} $\rho_\eps h e^{u_\eps} \rightharpoonup \ov{\rho}\; \delta_p$;
\item\label{l2} $u_\eps \stackrel{\eps\to 0}{\longrightarrow}-\infty$ uniformly in $\Omega$, $\forall\;\Omega \subset \subset \Sigma\bs \{p\}$;
\item\label{l3}  $\ov{u}_\eps \stackrel{\eps\to 0}{\longrightarrow}-\infty$;
\item\label{l4} There exist $\gamma\in (0,1)$, $s>2$ such that  $u_\eps-\ov{u_\eps} \stackrel{\eps\to 0}{\longrightarrow} \ov{\rho}\;G_p$  in $C^{0,\gamma}(\ov{\Omega})\cap W^{1,s}(\Omega)$ $\forall$ $\Omega \subset \subset \Sigma\bs\{p\}$;
\item\label{l5} $\nabla u_\eps$ is bounded in $L^q(\Sigma)$ $\forall\; q\in (1,2)$.
\end{enumerate}
\end{lemma}
\begin{proof}
\ref{l1}, \ref{l2} and \ref{l3} are direct consequences of proposition \ref{prop conc-comp}. To prove \ref{l4} we consider  Green's representation formula
$$
u_\eps(x)-\ov{u}_\eps =\rho_\eps \int_{\Sigma} G_x(y) h(y)e^{u_\eps(y)} dv_g(y).
$$
We stress that Green's function has the following properties:
\begin{itemize}
\item $\dis{|G_x(y)|\le C_1 (1+ |\log d(x,y)|)}$ $\forall\; x,y \in \Sigma$, $x\neq y$.
\item $\dis{|\nabla^x_g G_x(y)|}\le \frac{C_2}{d(x,y)}$ $\forall \; x,y\in \Sigma$, $x \neq y$. 
\item $G_x(y)= G_y(x)$ $\forall\; x,y\in \Sigma$, $x\neq y$.
\end{itemize}
Take $q>1$ such that $h\in L^q(\Sigma)$. The first property also yields
\begin{equation}\label{norma p' G}
\sup_{x\in \Sigma} \|G_x\|_{L^{q'}(\Sigma)}  \le C_3.
\end{equation}
Let us fix $\delta>0$ such that $B_{3\delta}(p)\subset \Sigma\bs \Omega$ and take a cut-off function $\ph$ such that $\ph \equiv 1$ in $B_{\delta}(p)$ and $\ph \equiv 0$ in $\Sigma\bs B_{2\delta}(p)$.
$$
u_\eps(x)-\ov{u_\eps} = \rho_\eps \int_{\Sigma} \ph(y) G_x(y) h(y)e^{u_\eps(y)} dv_g(y) +\rho_\eps\int_{\Sigma} (1-\ph(y)) G_x(y) h(y)e^{u_\eps(y)} dv_g(y).
$$
By \eqref{norma p' G} and \ref{l2} we have
$$
\left|\int_{\Sigma} (1-\ph(y)) G_x(y) {h}(y)e^{u_\eps(y)} dv_g(y)\right|\le \int_{\Sigma\bs B_\delta(p)}\left|G_x(y)\right|{h}(y)e^{u_\eps(y)} dv_g(y)\le$$
$$
\le  C_3  \|h\|_{L^q(\Sigma)} \|e^{u_\eps}\|_{L^\infty (\Sigma \bs B_\delta(p))} \stackrel{\eps\to 0}{\longrightarrow} 0. 
$$
By \ref{l1} and the smoothness of $\ph G_x$ for $x\in \ov{ \Omega}$ and $y\in \Sigma$ we get
$$ 
\int_{\Sigma} \ph(y) G_x(y) {h}(y)e^{u_\eps(y)} dv_g(y)   \stackrel{\eps\to 0}{\longrightarrow} \ph(p)G_x(p)=G_p(x)
$$
uniformly for $x \in \Omega$.
Similarly we have
$$
\nabla_gu_\eps(x)= \rho_\eps \int_{\Sigma} \ph(y) \nabla_g^x G_x(y) {h}(y)e^{u_\eps(y)} dv_g(y) +\rho_\eps\int_{\Sigma} (1-\ph(y)) \nabla_g^x G_x(y) {h}(y)e^{u_\eps(y)} dv_g(y)
$$
with
$$ 
\int_{\Sigma} \ph(y) \nabla^x_g G_x(y) {h}(y)e^{u_\eps(y)} dv_g(y) \stackrel{k\to \infty}{\longrightarrow} \nabla^x_g G_p(x)
$$
uniformly in $\Omega$ and, assuming $q\in (1,2)$, by the Hardy-Littlewood-Sobolev inequality
$$
\int_{\Sigma} \left(\int_{\Sigma} (1-\ph(y)) \nabla_g^x G_x(y) {h}(y)e^{u_\eps(y)} dv_g(y)\right)^s dv_g(x) \le $$
$$
\le  C_2^s \int_{\Sigma} \left(\int_{\Sigma\bs B_{\delta}(p)} \frac{{\;h}(y)e^{u_\eps(y)\;}}{d(x,y)} dv_g(y)\right)^s\; dv_g(x) \le  C  \|h\|_{L^q(\Sigma)}^s \|e^{u_n}\|_{L^\infty (\Sigma \bs B_\delta(p))}^s \stackrel{\eps\to 0}{\longrightarrow} 0
$$
where $$
\frac{1}{s} = \frac{1}{q} - \frac{1}{2}.
$$
Note that $q>1$ implies $s>2$. 
Finally, to prove \ref{l5}, we shall observe that for any $1<q<2$ there exists a positive constant $C_q$ such that 
$$
\int_{\Sigma}\ph \; dv_g=0 \quad \mbox{ and }\quad\int_{\Sigma} |\nabla_g \ph|^{q'}dv_g\le 1 \quad \Longrightarrow \quad \|\ph\|_{\infty} \le C_q.
$$
Hence $\forall\; \ph \in W^{1,q'}(\Sigma)$
$$
\int_\Sigma \nabla_g u_\eps \cdot \nabla_g \ph\; dv_g = -\int_\Sigma \Delta u_\eps \ph  \; dv_g \le C_q  \|\Delta u_\eps\|_{L^1(\Sigma) }\le \tilde{C}_q$$
so that
$$
\|\nabla u_\eps \|_{L^q}\le \sup\left\{ \int_{\Sigma}\nabla_g u_\eps \cdot \nabla_g \ph\; dv_g \;:\; \ph \in W^{1,q'}(\Sigma), \|\nabla \ph\|_{L^{q'}}\le1  \right\} \le \tilde{C}_q.
$$
\end{proof}

We now focus on the behaviour of $u_\eps$ near the blow-up point. First we consider the case $\alpha<0$. Let us fix a system of normal coordinates in a small ball $B_\delta(p)$, with $p$ corresponding to $0$ and $p_\eps$ corresponding to $x_\eps$. We define 
\begin{equation}\label{riscalamento}
\ph_\eps (x):= u_\eps(t_\eps x) -\lambda _\eps,\quad t_\eps:=e^{-\frac{\lambda_\eps}{2(1+\alpha)}}.
\end{equation}

\begin{lemma}
If $\alpha<0$, 
$\dis{
\frac{|x_\eps|}{t_\eps}} 
$ is bounded.
\end{lemma}
\begin{proof}
We define 
$$
\psi_\eps(x) = u_\eps(|x_\eps| x)+2(1+\alpha)\log |x_\eps| +s_\eps(|x_\eps|x)
$$
where $s_\eps(x)$ is the solution of
$$
\begin{Si}{cc}
-\Delta s_\eps = \frac{\rho_\eps}{|\Sigma|}  & \mbox{ in } B_\delta(0) \\
s_\eps =0 & \mbox{ if }  |x|=\delta
\end{Si}.
$$
The function $\psi_\eps$ satisfies 
$$
-\Delta \psi_\eps = |x_\eps|^{-2\alpha} \rho_\eps  h(|x_\eps| x) e^{-s_\eps(|x_\eps|x)} e^{\psi_\eps} = V_\eps e^{\psi_\eps}
$$
in $B_\frac{\delta}{|x_\eps|}(0)$.
We stress that, by standard elliptic estimates, $s_\eps$ is uniformly bounded in $C^1(\ov{B_\delta})$ and that  $G_p$ has the expansion
\begin{equation}\label{esp green}
G_p(x) = -\frac{1}{2\pi} \log |x|+ A(p) + O(|x|)
\end{equation}
in $B_\delta(0)$.  Thus
$$
|x_\eps|^{-2\alpha} h(|x_\eps| x)e^{-s_\eps(|x_\eps|x)} = |x_\eps|^{-2\alpha } e^{2\alpha \log (|x_\eps| |x|)-4\pi\alpha A(p)+O(|x_\eps| |x|)} e^{-s_\eps(|x_\eps|x)}  K(|x_\eps| x) \prod_{1\le i\le m,p_i\neq p} e^{-4\pi \alpha_i G_{p_i}(|x_\eps| x)}=$$
$$
= |x|^{2\alpha} e^{-4\pi \alpha A(p)} e^{O(|x_\eps| |x|)}e^{-s_\eps(|x_\eps|x)}  K(|x_\eps| x) \prod_{1\le i\le m,p_i\neq p} e^{-4\pi\alpha_i G_{p_i}(|x_\eps| x)}= |x|^{2\alpha} \tilde{h}(|x_\eps| x)
$$ 
where $\tilde{h}\in C^{1}(\ov{B_{\delta}})$. In particular $V_\eps$ is uniformly bounded in $C^1_{loc}(\R^2\bs\{0\})$. If there existed a subsequence such that $\dis{
\frac{|x_\eps|}{t_\eps}}\ra +\infty 
$
then 
$$\psi_\eps\left(\frac{x_\eps}{|x_\eps|}\right)= 2(1+\alpha)\log\left(\frac{|x_\eps|}{t_\eps}\right)+s_\eps(x_\eps)\ra +\infty,$$
so $\dis{y_0:= \lim_{\eps\to 0} \frac{x_\eps}{|x_\eps|}}$ would be a blow up point for $\psi_\eps$.  Since $y_0\neq 0$, applying proposition \ref{LiS} to $\psi_\eps$ in a small ball $B_r(y_0)$ we would get
$$
\liminf_{\eps \to 0}  \int_{B_r(y_0)} V_\eps e^{\psi_\eps} dx \ge 8\pi.
$$
But this would be in contradiction to (\ref{norm}) since 
$$
 \int_{B_r(y_0)} V_\eps e^{\psi_\eps} dx = \int_{B_{r(y_0)}}\rho_\eps\;|x_\eps|^{-2\alpha} h(|x_\eps| x) e^{-s_\eps(|x_\eps|x)} e^{\psi_\eps }dx  \le \rho_\eps \int_{B_\delta(p)} h e^{u_\eps} dv_g \le 8\pi(1+\alpha)<8\pi.
$$
\end{proof}

\begin{lemma}\label{lemma risc}
Assume $\alpha<0$. Then, possibly passing to a subsequence, $\ph_\eps$ converges uniformly on compact subsets of $\R^2$ and in $H^1_{loc}(\R^2)$ to 
$$
\ph_0(x):= -2\log\left(1+\frac{\pi c(p) }{1+\alpha}|x|^{2(1+\alpha)}\right)
$$ 
where  $\dis{c(p)= K(p)e^{-4\pi \alpha A(p)} \prod_{1\le i\le m, p_i\neq p} e^{-4\pi \alpha_i G_{p_i}(p)}}$. 
\end{lemma}
\begin{proof}
The function $\ph_\eps$ is defined in $B_\eps=B_\frac{\delta}{t_\eps}(0)$ and satisfies
$$
-\Delta \ph_\eps = t_\eps^2 \rho_\eps \left(h(t_\eps x) e^{\ph_\eps} e^{\lambda_\eps}- \frac{1}{|\Sigma|}\right)= t_\eps^{-2\alpha} \rho_\eps  h(t_\eps x) e^{\ph_\eps}- \frac{t_\eps^2 \rho_\eps}{|\Sigma|}
$$
and
\begin{equation*}
t_\eps^{-2\alpha}\int_{B_\frac{\delta}{t_\eps}} h(t_\eps x) e^{\ph_\eps} \le 1.
\end{equation*}
As in the previous proof we have
$$
t_\eps^{-2\alpha} h(t_\eps x) = t_\eps^{-2\alpha } e^{2\alpha \log (t_\eps |x|)-4\pi\alpha A(p)+O(t_\eps |x|)} K(t_\eps x) \prod_{1\le i\le m,p_i\neq p} e^{-4\pi \alpha_i G_{p_i}(t_\eps x)} =$$
$$
= |x|^{2\alpha} e^{-4\pi\alpha A(p)} e^{O(t_\eps |x|)} K(t_\eps x) \prod_{1\le i\le m,p_i\neq p} e^{-4\pi\alpha_i G_{p_i}(t_\eps x)}\stackrel{\eps \to 0}{\longrightarrow } c(p)|x|^{2\alpha}
$$
in $L^q_{loc}(\R^2)$ for some $q>1$.  Fix $R>0$ and let $\psi_\eps$ be the solution of 
$$
\begin{Si}{cc}
-\Delta \psi_\eps =  t_\eps^{-2\alpha} \rho_\eps  h(t_\eps x) e^{\ph_\eps}- \frac{t_\eps^2 \rho_\eps}{|\Sigma|} & \mbox{ in } B_R(0)  \\
\psi_\eps =0 & \mbox{ su } \partial B_R(0)
\end{Si}.
$$
Since $\Delta \psi_\eps$ is bounded in $L^q(B_R(0))$ with $q>1$, elliptic regularity shows that $\psi_\eps$ is bounded in $W^{2,q}(B_R(0))$ and by Sobolev's embeddings we may extract a subsequence such that $\psi_\eps$ converges in $H^1(B_R(0))\cap C^{0,\lambda}(B_R(0))$.   The function $\xi_\eps = \ph_\eps -\psi_\eps$ is harmonic in $B_R$ and bounded from above. Furthermore $\xi_\eps\left(\frac{x_\eps}{t_\eps}\right)= -\psi_\eps\left(\frac{x_\eps}{t_\eps}\right)$ is bounded from below, hence  by Harnack inequality $\xi_\eps$ is uniformly bounded in $C^{2}(\ov{B_{\frac{R}{2}}}(0))$. Thus $\ph_\eps$ is bounded in $W^{2,q}(B_{\frac{R}{2}})$ and we can extract a subsequence converging in $H^1(B_{\frac{R}{2}})\cap C^{0,\lambda}(B_\frac{R}{2})$. Using a diagonal argument we find a subsequence for which
$\ph_\eps$ converges in $H^1_{loc}(\R^2)\cap C^{0,\lambda}_{loc}(\R^2)$ to a function $\ph_0$ solving
$$
-\Delta\ph_0 = 8\pi (1+\alpha) c(p)|x|^{2\alpha} e^{\ph_0}
$$ 
on $\R^2$ with 
$$
\int_{\R^2} |x|^{2\alpha}e^{\ph_0(x)} dx <\infty.
$$
The classification result in \cite{praj-tar} yields
$$
\ph_0(x)=-2\log\left(1+\frac{\pi e^{\lambda} c(p)}{1+\alpha} |x|^{2(1+\alpha)}\right)+\lambda
$$
for some $\lambda\in \R$. To conclude the proof it remains to note that, since $0$ is the unique maximum point of $\ph_0$, the uniform convergence of $\ph_\eps$ implies $\frac{x_\eps}{t_\eps}\ra 0$ and   $\lambda=0$.
\end{proof}

As in \cite{djlw}, to give a lower bound on $J_\eps (u_\eps)$ we need the following estimate from below for  $u_\eps$:

\begin{lemma}\label{lemma stima}
Fix $R>0$ and define $r_\eps= t_\eps R$. If $\alpha<0$ and $u_\eps$ satisfies (\ref{eq desing}), (\ref{norm}), (\ref{blowup}), then 
$$
u_\eps \ge \ov{\rho}\; G_p-\lambda_\eps-\ov{\rho}\; A(p) + 2 \log\left(\frac{ R^{2(1+\alpha)}}{1+\frac{\pi c(p)}{1+\alpha} R^{2(1+\alpha)}}\right) +o_\eps(1)
$$
in $ \Sigma \bs B_{r_\eps}(p)$.
\end{lemma}
\begin{proof}
$\forall\; C>0$ we have 
$$
-\Delta_g (u_\eps-\ov{\rho}\; G_p-C) = \rho_\eps \left( h e^{u_\eps }-\frac{1}{|\Sigma|}\right) +\frac{\ov{\rho}}{|\Sigma|} = \rho_\eps h e^{u_\eps} +\frac{\eps}{|\Sigma|}\ge 0.
$$
Let us consider normal coordinates near $p$. We know that 
$$
G_p(x)=-\frac{1}{2\pi} \log |x| + A(p)+O(|x|),
$$
so by lemma \ref{lemma risc} if $x=t_\eps y$ with $|y| =R$ we have
$$
u_\eps (x)-\ov{\rho}\; G_p= \ph_\eps(y)+\lambda_\eps  +4(1+\alpha)\log (t_\eps R) -\ov{\rho}A(p)+o_\eps(1)= 
$$
$$
= -2\log\left(1+\frac{\pi c(p)}{1+\alpha}R^{2(1+\alpha)}\right)-\lambda_\eps +\log{R^{4(1+\alpha)}}-\ov{\rho}\; A(p)+o_\eps(1).
$$
Thus, taking
$$C_\eps=- \lambda_\eps-\ov{\rho}\; A(p) +2 \log\left(\frac{R^{2(1+\alpha)}}{1+\frac{\pi  c(p)}{1+\alpha} R^{2(1+\alpha)}}\right) +o_\eps(1)
$$
we have $u_\eps-\ov{\rho} G_p -C_\eps\ge 0$ on $\partial B_{r_\eps}(p)$ 
and the conclusion follows from the maximum principle. 
\end{proof}

As a consequence we also have

\begin{lemma}\label{lemma limiti}
$t_\eps^2 \ov{u}_\eps \longrightarrow 0$.
\end{lemma}
\begin{proof}
By lemma \ref{lemma risc} 
$$
\int_{B_{t_\eps}(p)} u_\eps\; dv_g = t_\eps^2 \int_{B_1(0)} \ph_\eps(y)dy +\lambda_\eps |B_{t_\eps}|= o_\eps(1).
$$
and by the previous lemma
$$
\lambda_\eps |\Sigma|\ge\int_{\Sigma \bs B_{t_\eps}(p)} u_\eps \ge \ov{\rho} \int_{\Sigma\bs B_{t_\eps}(p)} G_p \;dv_g - \lambda_\eps |\Sigma\bs B_{t_\eps}(p)| +O(1).
$$
Thus $\dfrac{|\ov{u}_\eps|}{\lambda_\eps}$ is bounded and, since $\lambda_\eps t_\eps^2 =o_\eps(1)$, we get the conclusion.
\end{proof}

The case $\alpha=0$ can be studied in a similar way. The main difference is that, since we do not know whether  $\frac{|x_\eps|}{t_\eps}$ is bounded, we have to center the scaling in $p_\eps$ and not in $p$. Note that $\beta(p)=0$ means that $p\in \Sigma\bs S$ is a regular point of $h$.

\begin{lemma}\label{caso0}
Assume that $\alpha=0$ and that $u_\eps$ satisfies \eqref{eq desing}, \eqref{norm} and \eqref{blowup}. In normal coordinates near $p$ define
$$
\psi_\eps (x)= u_\eps (x_\eps + t_\eps x)-\lambda_\eps \quad \mbox{ where } \quad  t_\eps= e^{-\frac{\lambda_\eps}{2}}.
$$
Then
\begin{enumerate}
\item $\psi_\eps$ converges in $C^1_{loc}(\R^2)$ to $$\psi_0(x)= -2\log(1+\pi h(p) |x|^2)$$ 
\item $\forall \;R>0$ one has 
 $$u_\eps\ge 8\pi G_{p_\eps} - \lambda_\eps -8\pi A(p) + 2\log\left(\frac{R^2}{1+\pi h(p) R^2}\right)+o_\eps(1)$$ in $\Sigma \bs B_{R t_\eps}(p_\eps);$
\item  $t_\eps^2 \ov{u}_\eps\to 0$.
\end{enumerate}
\end{lemma}

\section{A lower bound}
In this section and in the next one we present the proof of theorem \ref{teo 1}. We begin by giving an estimate from below of $\dis{\inf_{H^1(\Sigma)}J}$.  As before we consider $u_\eps$ satisfying \eqref{minim}, \eqref{eq desing}, \eqref{norm}, and (\ref{blowup}). Again we will focus on the case $\alpha<0$ since the computation for $\alpha=0$ is equivalent to the one in \cite{djlw}. We consider normal coordinates in a small ball $B_\delta (p)$ and assume that $G_p$ has the expansion (\ref{esp green}) in $B_\delta(p)$. Let $t_\eps$ be defined as in \eqref{riscalamento}, then $\forall \;R>0$ we shall consider the decomposition

$$\dis{\int_{\Sigma}|\nabla_g u_\eps|^2dv_g= \int_{\Sigma\bs B_\delta(p)} |\nabla_g u_\eps|^2dv_g + \int_{B_\delta\bs B_{r_\eps}(p)} |\nabla_g u_\eps|^2dv_g + \int_{B_{r_\eps}(p)}|\nabla_g u_\eps|^2dv_g}.$$

On $\Sigma\bs B_\delta(p)$ we can use lemma \ref{lemma riass} and an integration by parts to obtain:

\begin{eqnarray}\label{esterno}
\int_{\Sigma\bs B_\delta} |\nabla_g u_\eps|^2 dv_g &=&  \ov{\rho}^2\int_{\Sigma\bs B_\delta}|\nabla_g G_p|^2 dv_g +o_\eps(1)= \nonumber \\
 &=&-\frac{\ov{\rho}^2}{|\Sigma|} \int_{\Sigma\bs B_\delta}   G_p\; dv_g -  \ov{\rho}^2\int_{\partial B_\delta}  G_p \D{G_p}{n} \; d\sigma_g+o_\eps(1)= \nonumber \\
 &=&- \ov{\rho}^2\int_{\partial B_\delta}   G_p \D{G_p}{n} d\sigma_g +o_\eps(1)+o_\delta (1).
\end{eqnarray}

On $B_{r_\eps}(p)$ the convergence result for the scaling (\ref{riscalamento}) stated  in lemma \ref{lemma risc} yields 

\begin{equation}\label{interno}
\int_{B_{r_\eps}} |\nabla_g u_\eps|^2 dv_g = \int_{B_R(0)} |\nabla \ph_0|^2 dx +o_\eps(1)= 2\ov{\rho}\left(\log\left(1+\frac{\pi \;c(p)}{1+\alpha} R^{2(1+\alpha)}\right)-1\right)+o_\eps(1)+o_R(1).
\end{equation}

For the remaining term we can use (\ref{eq desing}) and lemma \ref{lemma riass} to obtain

\begin{eqnarray}\label{anello}
\int_{B_\delta \bs B_{r_\eps}} \hspace{-0.2cm} |\nabla_g u_\eps|^2 dv_g \hspace{-0.1cm} &=&   \hspace{-0.1cm}\rho_\eps \int_{B_\delta \bs B_{r_\eps}}  \hspace{-0.2cm}  h  e^{u_\eps} u_\eps dv_g  - \frac{\rho_\eps}{|\Sigma|}\int_{{B_\delta \bs B_{r_\eps}}}   \hspace{-0.2cm} u_\eps dv_g+ \int_{\partial B_\delta } u_\eps \D{u_\eps}{n} d\sigma_g- \int_{\partial { B_{r_\eps}}}  u_\eps \D{u_\eps }{n} d\sigma_g = \nonumber \\
&=&  \hspace{-0.1cm} \rho_\eps \int_{B_\delta \bs B_{r_\eps}}  \hspace{-0.2cm}  h  e^{u_\eps} u_\eps dv_g - \frac{\rho_\eps}{|\Sigma|}\int_{{B_\delta \bs B_{r_\eps}}}   \hspace{-0.2cm}  u_\eps dv_g + \ov{u}_\eps \int_{\partial B_\delta } \D{u_\eps}{n} d\sigma_g - \int_{\partial { B_{r_\eps}}}  u_\eps \D{u_\eps }{n} d\sigma_g  + \nonumber \\
& & + \ov{\rho}^2\int_{\partial B_\delta} G_p \D{G_p}{n}d\sigma_g +o_\eps(1).
\end{eqnarray}

By lemma \ref{lemma stima} and (\ref{norm}) we get
\begin{eqnarray}\label{parte1}
 \rho_\eps \int_{B_\delta \bs B_{r_\eps}}  h e^{u_\eps} u_\eps dv_g &\ge& \rho_\eps  \ov{\rho} \int_{B_\delta \bs B_{r_\eps}}   h e^{u_\eps} G_p  dv_g - \rho_\eps \lambda_\eps \int_{B_\delta \bs B_{r_\eps}}  h e^{u_\eps} dv_g
 +O_R(1)\rho_\eps \int_{B_\delta \bs B_{r_\eps}}   h e^{u_\eps} dv_g = \nonumber \\
&=& \rho_\eps  \ov{\rho} \int_{B_\delta \bs B_{r_\eps}}   h e^{u_\eps} G_p dv_g - \rho_\eps \lambda_\eps \int_{B_\delta \bs B_{r_\eps}}  h e^{u_\eps} dv_g  +o_\eps(1).
\end{eqnarray}

Again by (\ref{eq desing}) and lemma \ref{lemma riass}

\begin{eqnarray}\label{parte2}
\rho_\eps \int_{B_\delta \bs B_{r_\eps}}  \hspace{-0.4cm}  h e^{u_\eps} G_p dv_g  \hspace{-0.14cm} &=&  \hspace{-0.25cm} \int_{B_\delta \bs B_{r_\eps}} G_p \left(-\Delta u_\eps  +\frac{\rho_\eps}{|\Sigma|}\right) dv_g = \nonumber \\&=&  \hspace{-0.2cm} - \frac{1}{|\Sigma|}\int_{B_\delta \bs B_{r_\eps}}   \hspace{-0.5cm} u_\eps dv_g +\hspace{-0.1cm}\int_{\partial  B_\delta}  \hspace{-0.2cm} u_\eps \D{G_p}{n} \hspace{-0.05cm} - G_p \D{u_\eps}{n} d\sigma_g +\hspace{-0.1cm}\int_{\partial  B_{r_\eps}}  \hspace{-0.3cm} G_p \D{u_\eps}{n}  \hspace{-0.05cm}- u_\eps \D{G_p}{n} d\sigma_g  \hspace{-0.02cm} +  \hspace{-0.01cm}o_\delta(1)= \nonumber \\
&=&  \hspace{-0.25cm}- \frac{1}{|\Sigma|}\int_{B_\delta \bs B_{r_\eps}}  \hspace{-0.3cm}  u_\eps dv_g +\ov{u}_\eps\int_{\partial  B_\delta}  \D{G_p}{n} d\sigma_g +\int_{\partial  B_{r_\eps}}  G_p \D{u_\eps}{n} d\sigma_g - \int_{\partial  B_{r_\eps}}   u_\eps \D{G_p}{n} d\sigma_g +\nonumber \\
 & &\hspace{-0.25cm} + \;o_\eps(1) +o_\delta(1),
\end{eqnarray}
and
\begin{eqnarray}\label{parte3}
\rho_\eps \lambda_\eps \int_{B_\delta \bs B_{r_\eps}} h e^{u_\eps} dv_g &=&    -  \lambda_\eps\int_{\partial B_\delta \bs B_{r_\eps}} \D{u_\eps}{n} d\sigma_g  +\frac{\rho_\eps \lambda_\eps}{|\Sigma|} \left(Vol(B_\delta)-Vol(B_{r_\eps})\right) =\nonumber \\
 &=& -\lambda_\eps \int_{\partial B_\delta} \D{u_\eps}{n} d\sigma_g + \lambda_\eps\int_{\partial B_{r_\eps}} \D{u_\eps}{n} d\sigma_g+ \frac{\rho_\eps  \lambda_\eps}{|\Sigma|} Vol(B_\delta)+o_\eps(1).
\end{eqnarray}

Using  \eqref{anello}, \eqref{parte1}, \eqref{parte2} and \eqref{parte3} we get

\begin{eqnarray}\label{anello2}
\int_{B_\delta \bs B_{r_\eps}}|\nabla_g u_\eps|^2dv_g &\ge& - (16\pi(1+\alpha)-\eps) \frac{1}{|\Sigma|}\int_{B_\delta \bs B_{r_\eps}} u_\eps  \; dv_g  - \frac{\rho_\eps\lambda_\eps}{|\Sigma|}Vol(B_\delta)+ \nonumber \\
 &+&   \ov{\rho}\;  \ov{u}_\eps \int_{\partial B_\delta}\D{G_p}{n} d\sigma_g+ \lambda_\eps \int_{\partial B_\delta} \D{u_\eps}{n} d\sigma_g +  \ov{u}_\eps \int_{\partial B_{\delta}} \D{u_\eps}{n} d\sigma_g +  \\
 &+&  \ov{\rho}^2 \int_{\partial B_\delta} G_p \D{ G_p}{n} d\sigma_g  -\ov{\rho}\int_{\partial B_{r_\eps}} u_\eps \D{G_p}{n} d\sigma_g -\int_{\partial B_{r_\eps}} \Big( u_\eps
 - \ov{\rho}\;G_p +\lambda_\eps \Big)\D{u_\eps}{n} +\nonumber \\
 &+&o_\eps(1)+o_\delta(1) \nonumber .
\end{eqnarray}

By lemmas \ref{lemma riass} and \ref{lemma limiti} we can say that 
$$
\int_{B_\delta \bs B_{r_\eps}}u_\eps dv_g =  \int_{B_\delta \bs B_{r_\eps}}  (u_\eps-\ov{u}_\eps) dv_g + \ov{u}_\eps (Vol (B_\delta)-Vol(B_{r_\eps}))= \ov{u}_\eps Vol(B_\delta) + o_\delta(1)+o_\eps(1).
$$

Using Green's formula
$$
\ov{u}_\eps \int_{\partial B_\delta} \D{G_p}{n} d\sigma_g  =   - \ov{u}_\eps\int_{\Sigma\bs B_\delta} \Delta_g G_p \; dv_g  =- \ov{ u}_\eps \left(1-\frac{Vol(B_\delta)}{|\Sigma|}\right)
.$$

Similarly

$$
\int_{\partial B_\delta}\D{u_\eps}{n}d\sigma_g=  -\int_{\Sigma\bs B_\delta} \Delta u_\eps \; dv_g  = \int_{\Sigma\bs B_\delta} \rho_\eps \left( h e^{u_\eps}-\frac{1}{|\Sigma|}\right)dv_g \ge  -\rho_\eps \left(1-\frac{Vol(B_\delta)}{|\Sigma|}\right)
$$
and
\begin{eqnarray*}
\ov{u}_\eps\int_{\partial B_\delta} \D{u_\eps}{n}d\sigma_g &=&   \ov{u}_\eps \rho_\eps e^{\ov{u}_\eps} \int_{\Sigma\bs B_\delta (p)}  h \;e^{u_\eps -\ov{u}_\eps}dv_g-\ov{u}_\eps \rho_\eps \left(1-\frac{Vol(B_\delta)}{|\Sigma|}\right)=\\
&=&- \ov{u}_\eps \rho_\eps \left(1-\frac{Vol(B_\delta)}{|\Sigma|}\right) +o_\eps(1).
\end{eqnarray*}

Lemma \ref{lemma risc} yields

\begin{eqnarray*}
\int_{\partial B_{r_\eps}} u_\eps\D{G_p}{n} d\sigma_g &=&  \lambda_\eps \int_{\partial B_\eps}\D{G_p}{n} d\sigma_g + t_\eps\int_{\partial B_R(0)}\ph_\eps \D{G_p}{n}(t_\eps x)(1+o_\eps(1))d\sigma=\\
&=& - \lambda_\eps \left(1-\frac{Vol(B_{r_\eps})}{|\Sigma|}\right)  +t_\eps \int_{\partial B_R(0)} \ph_0 \left(-\frac{1}{2\pi t_\eps R}  +O(1)\right) d\sigma  =\\
&=&-\lambda_\eps +2  \log\left(1+ \frac{\pi\; c(p)}{1+\alpha} R^{2(1+\alpha)}\right)+ o_\eps(1)
\end{eqnarray*}

and the estimate in lemma \ref{lemma stima} gives

$$
-\int_{\partial B_{r_\eps}} \Big( u_\eps - \ov{\rho}\;G_p +\lambda_\eps \Big)\D{u_\eps}{n}  d\sigma_g \ge $$
$$
\ge  \left(2 \log\left(\frac{ R^{2(1+\alpha)}}{1+ \frac{ \pi c(p)}{(1+\alpha)} R^{2(1+\alpha)}}\right) - \ov{\rho}  A(p)\right) \frac{8 \pi^2  c(p) R^{2 (1+\alpha)} }{ \left(1+\frac{\pi  c(p) R^{2 (1+\alpha)}
   }{1+\alpha}\right)} +o_\eps(1) =
   $$
   $$
   = -  \ov{\rho}^2A(p) - 2 \; \ov{\rho} \; \log\left(\frac{\pi c(p)}{1+\alpha}\right) +o_\eps(1) +o_R(1).
$$

Hence

\begin{eqnarray}\label{anello finale}
\int_{B_\delta \bs B_{r_\eps}}|\nabla_g u_\eps|^2dv_g &\ge&-(16\pi(1+\alpha)-\eps)\ov{u}_\eps +\eps\lambda_\eps + \ov{\rho}^2 \int_{\partial B_\delta} G_p \D{G_p}{n} d\sigma_g +\nonumber \\
&-&   2 \ov{\rho}  \log \left(1+\frac{\pi c(p)}{1+\alpha} R^{2(1+\alpha)}\right) - \ov{\rho}^2A(p) - 2 \ov{\rho}\log\left(\frac{\pi c(p)}{1+\alpha}\right) +\nonumber \\
&+ & o_\eps(1)+o_\delta(1)+o_R(1).
\end{eqnarray}

By \eqref{esterno}, \eqref{interno} and \eqref{anello finale} we can therefore conclude 
\begin{eqnarray*}
\int_{\Sigma}|\nabla_g u_\eps|^2 dv_g 
&\ge & -(16\pi(1+\alpha)-\eps)\ov{u}_\eps +\eps \lambda_\eps  -\ov{\rho}^2A(p)-2\ov{\rho}\log\left(\frac{\pi c(p)}{1+\alpha}\right) -2\ov{\rho}+ \\
 & & + \;o_\eps(1)+o_\delta(1) + o_R(1),
\end{eqnarray*}

so that 

\begin{eqnarray*}
J_\eps (u_\eps)&\ge &  \frac{\eps}{2}(\lambda_\eps-\ov{u}_\eps) -\frac{\ov{\rho}^2}{2}A(p)-\ov{\rho}\log\left(\frac{\pi c(p)}{1+\alpha}\right)-\ov{\rho}+\rho_\eps \log |\Sigma| +  o_\eps(1)+ o_\delta(1) + o_R(1) \ge \\
&\ge& -\ov{\rho} \left( 4\pi(1+\alpha) A(p)+1+\log \left(\frac{\pi c(p)}{1+\alpha}\right)-\log|\Sigma|\right) +o_\eps(1)+o_\delta(1) + o_R(1).
\end{eqnarray*}

As $\eps,\delta\to 0$ and $R\to \infty$ we obtain
\begin{eqnarray}\label{stima basso}
\inf_{H^1(\Sigma)} J&\ge& -\ov{\rho} \left( 4\pi(1+\alpha) A(p)+1+\log \left(\frac{\pi c(p)}{1+\alpha}\right)-\log{|\Sigma|}\right)=\\
&=& -\ov{\rho}\left( 1+\log\frac{\pi}{|\Sigma|} +4\pi A(p)+\log  \left(\frac{K(p)}{1+\alpha}\prod_{q\in S, q\neq p} e^{-4\pi\beta(q) G_q(p)}\right)\right). \nonumber
\end{eqnarray}

Using lemma \ref{caso0} it is possible to prove that (\ref{stima basso}) holds even for $\alpha =0$. About the blow-up point $p$ we only know that $\beta(p)=\alpha$, so  we have proved 
\begin{prop}\label{prop1}
If $J$ has no minimum point, then
$$
\inf_{H^1(\Sigma)} J \ge-\ov{\rho} \left( 1+\log\frac{\pi}{|\Sigma|}+\max_{p\in \Sigma ,\beta(p)=\alpha}  \left\{4\pi A(p)+ \log  \left(\frac{K(p)}{1+\alpha}\prod_{q\in S, q\neq p} e^{-4\pi\beta(q) G_q(p)}\right)\right\}\right).
$$
\end{prop}

Notice that, if $\alpha<0$, the set
$$
\left\{ p \in \Sigma \;:\; \beta(p)=\alpha \right\}=\left\{p_i\;:\; i\in\{1,\ldots,m\},\;\alpha_i=\alpha\right\}
$$
is finite, while if $\alpha =0$
$$
\left\{ p \in \Sigma \;:\; \beta(p)=\alpha\right\}=\Sigma\bs S.
$$
Although this set is not finite, the maximum in the above expression is still well defined since the function 
$$
p\longmapsto 4\pi A(p)+ \log  \left(K(p)\prod_{q\in S} e^{-4\pi\beta(q) G_q(p)}\right) = 4\pi A(p)+ \log   h(p)
$$ 
is continuous on $\Sigma\bs S$ and approaches $-\infty$ near $S$.

\section{An estimate from above}
In order to complete the proof of theorem \ref{teo 1}  we need to exhibit a sequence  $\ph_\eps \in H^1(\Sigma)$ such that  
$$
J(\ph_\eps)\longrightarrow  -\ov{\rho} \left( 1+\log\frac{\pi}{|\Sigma|}+\max_{p\in \Sigma ,\beta(p)=\alpha} \left\{ 4\pi A(p)+ \log  \left(\frac{K(p)}{1+\alpha}\prod_{q\in S, q\neq p} e^{-4\pi\beta(q) G_q(p)}\right)\right\}\right) 
$$
Let us define $r_\eps:=\gamma_\eps \eps^\frac{1}{2(1+\alpha)}$ where $\gamma_\eps$ is chosen so that
\begin{equation}\label{condizioni}
\gamma_\eps\to+\infty,\quad r_\eps^2 \log\eps\ra 0, \quad r_\eps^2 \log(1+\gamma_\eps^{2(1+\alpha)})\ra  0.
\end{equation}
Let $p\in \Sigma$ be such that $\beta(p)=\alpha$ and 
$$4\pi A(p)+ \log  \left(\frac{K(p)}{1+\alpha}\prod_{q\in S, q\neq p} e^{-4\pi\beta(q) G_q(p)}\right)=\max_{\xi \in \Sigma ,\beta(\xi)=\alpha} \left\{4\pi A(\xi)+ \log  \left(\frac{K(\xi)}{1+\alpha}\prod_{q\in S, q\neq \xi} e^{-4\pi\beta(q) G_q(\xi)}\right)\right\}$$
and consider a cut-off function $\eta_\eps$ such that $\eta_\eps \equiv 1$ in $B_{r_\eps}(p)$, $\eta_\eps \equiv 0$ in $\Sigma\bs B_{2 r_\eps}(p)$ and  $|\nabla_g \eta_\eps|=O(r_\eps^{-1})$. 
Define 
$$
\ph_\eps (x)= \begin{Si}{cc}
-2\log(\eps+r^{2(1+\alpha)})+\log \eps & r\le r_\eps \\
\ov{\rho} \left( G_p-\eta_\eps \sigma \right) +C_\eps +\log \eps & r\ge   r_\eps
\end{Si}
$$
where $r=d(x,p)$,  $\sigma(x)=O(r)$ is defined by
\begin{equation}\label{green beta}
G_p(x) = -\frac{1}{2\pi} \log r +A(p) + \sigma(x),
\end{equation}
and 
$$
C_\eps= -2\log\left(\frac{1 +\gamma_\eps^{2(1+\alpha)}}{\gamma_\eps^{2(1+\alpha)}}\right) -\ov{\rho} \; A(p).
$$ 
In the case $\alpha_i=0$ $\forall\;i$, a similar family of functions was used in \cite{djlw} to give an existence result for \eqref{eq nonsing} by proving, under some strict assumptions on $h$, that
$$
\inf_{H^1(\Sigma)} J_{\ov{\rho}}< -8\pi\left( 1+\log\left(\frac{\pi}{|\Sigma|}\right)  +\max_{p\in \Sigma} \left\{4\pi A(p)+ \log h(p)\right\}\right).
$$ 
Here we only prove large inequality but we have no extra assumptions on $h$.
Taking normal coordinates in a neighbourhood of $p$ it is simple to verify that
\begin{eqnarray*}
\int_{B_{r_\eps}}|\nabla_g \ph_\eps|^2dv_g 
&=&16\pi(1+\alpha)\left(\log\left(1+\gamma_\eps^{2(1+\alpha)}\right) + \frac{1}{1+\gamma_\eps^{2(1+\alpha)}}-1\right)+o_\eps(1)= \\
&=&16\pi(1+\alpha)\left(\log\left(1+\gamma_\eps^{2(1+\alpha)}\right) -1\right)+o_\eps(1).
\end{eqnarray*}

By our definition of $\ph_\eps$
$$
\int_{\Sigma\bs B_{r_\eps}}|\nabla_g \ph_\eps|^2dv_g=\ov{\rho}^2\left(\int_{\Sigma\bs B_{r_\eps}} |\nabla_g G_p|^2dv_g + \int_{\Sigma\bs B_{r_\eps}} |\nabla_g (\eta _\eps\sigma)|^2dv_g  - 2\int_{\Sigma\bs B_{r_\eps}}  \nabla_g G_p \cdot \nabla_g (\eta_\eps \sigma)\; dv_g
\right)
$$
and by the properties of $\eta_\eps$ 
$$
 \int_{\Sigma\bs B_{r_\eps}}  |\nabla_g (\eta_\eps \sigma)|^2 dv_g= \int_{B_{2r_\eps}\bs B_{r_\eps}} |\nabla_g \eta_\eps|^2 \sigma^2  + 2 \eta_\eps \sigma\; \nabla_g \eta_\eps \cdot \nabla_g \sigma + \eta_\eps^2|\nabla_g \sigma|^2 \; dv_g= O(r_\eps^2).
$$
Hence, integrating by parts and using (\ref{green beta}), one has
\begin{eqnarray*}
\int_{\Sigma\bs B_{r_\eps}}|\nabla_g \ph_\eps|^2dv_g  &=&  \ov{\rho} ^2 \left(\int_{\Sigma\bs B_{r_\eps}} |\nabla G_p|^2 dv_g - 2\int_{\Sigma\bs B_{r_\eps}} \nabla_g G_p \cdot \nabla_g ( \eta_\eps \sigma) \;dv_g
\right)+o_\eps(1)=\\
&=& - \ov{\rho} ^2 \left( \frac{1}{|\Sigma|} \int_{\Sigma\bs B_{r_\eps}} (G_p-2\eta_\eps \sigma) \; dv_g +\int_{\partial B_{r_\eps}} (G_p-2\eta_\eps \sigma) \D{ G_p}{n} d\sigma_g  \right)+o_\eps(1)=\\
&=&  -\ov{\rho} ^2 \int_{\partial B_{r_\eps}} (G_p-2\sigma )\D{G_p}{n} d\sigma_g +o_\eps(1)=\\
&=&  -\ov{\rho}^2 \int_{\partial B_{r_\eps}} \hspace{-0.08cm}\left(-\frac{1}{2\pi}\log(r_\eps) +A(p)-\sigma \right)\left( -\frac{1}{2\pi r_\eps}+   \nabla \sigma \right)(1+O(r_\eps^2))d\sigma +o_\eps(1)=\\
&=& -\ov{\rho} ^2 \int_{\partial B_{r_\eps}} \left( \frac{\log r_\eps}{4\pi^2 r_\eps }-\frac{1}{2\pi r_\eps} A(p) +O(\log r_\eps) + O(1)  \right)d\sigma +o_\eps(1)=\\
&=&  -\frac{\ov{\rho} ^2 }{2\pi} \log (\gamma_\eps \eps^\frac{1}{2(1+\alpha)}) +\ov{\rho}^2 A(p) +o_\eps (1)=\\
&=&   -2\ov{\rho} \left( \log \gamma_\eps^{2(1+\alpha)} + \log \eps - 4 \pi (1+\alpha)A(p)\right) +o_\eps (1).
\end{eqnarray*}

Thus 
\begin{eqnarray}\label{grad}
\int_{\Sigma}|\nabla_g \ph_\eps|^2 dv_g  &=& 2\ov{\rho}\left(\log\left( \frac{1+\gamma_\eps^{2(1+\alpha)}}{\gamma_\eps^{2(1+\alpha)}}\right) -1+ 4\pi(1+\alpha)A(p)- \log \eps \right)+o_\eps(1)=\nonumber \\
&=&-2\ov{\rho} \left( 1- 4\pi(1+\alpha)A(p)+ \log \eps \right) +o_\eps(1).
\end{eqnarray}

Similarly one has 

\begin{eqnarray*}
\int_{B_{r_\eps}} \ph_\eps  \; dv_g &=& |B_{r_\eps}| \log \eps-4\pi\int_{0}^{r_\eps} r \log \left(\eps + r^{2(1+\alpha)}\right)(1+o_\eps(1))dr   =\\
&=&|B_{r_\eps}| \log \eps -2\pi r_\eps^2 \log \eps  -4\pi \int_0^{r_\eps} r \log\left(1+\frac{r^{2(1+\alpha)}}{\eps}\right)(1+o_\eps(1))dr = \\
&=&O(r_\eps^2 \log \eps ) -4\pi \int_0^{1} r_\eps^2  s \log \left( 1+\gamma_\eps^{2(1+\alpha)} s^{2(1+\alpha)}\right)(1+o_\eps(1))dr=\\
&=&O(r_\eps^2 \log \eps) + O(r_\eps^2 \log(1+\gamma_\eps^{2(1+\alpha)}))=o_\eps(1) 
\end{eqnarray*}
and
\begin{eqnarray*}
\int_{\Sigma\bs B_{r_\eps}} \ph_\eps  \; dv_g &=&  \ov{\rho}  \int_{\Sigma\bs B_{r_\eps}} (G_p-\eta_\eps \sigma) dv_g+( C_\eps +\log\eps) |\Sigma\bs B_{r_\eps}(p)|=\\
&=&|\Sigma| \log \eps -\ov{\rho} |\Sigma|A(p) +o_\eps(1)
\end{eqnarray*}
so that
\begin{equation}\label{media}
\frac{1}{|\Sigma|}\int_{\Sigma}\ph_\eps dv_g = \log \eps -\ov{\rho}\; A(p)+o_\eps(1).
\end{equation}

To compute the integral of the exponential term we fix  a small $\delta>0$ and observe that

\begin{equation*}
 \int_\Sigma  h e^{\ph_\eps} dv_g = \tilde{h}(p) \int_{B_{r_\eps}}  e^{-4\pi\alpha G_p} e^{\ph_\eps} dv_g+ \int_{B_{r_\eps}}\left(\tilde{h}-\tilde{h}(p)\right)  e^{-4\pi\alpha G_p} e^{\ph_\eps} dv_g + \int_{B_\delta\bs B_{r_\eps}} h e^{\ph_\eps}dv_g+ \int_{\Sigma\bs B_\delta} h e^{\ph_\eps} dv_g
\end{equation*}

where $\dis{\tilde{h} = h\; e^{4\pi \alpha G_p}= K \prod_{q\in S,q\neq p}}e^{-4\pi \beta(q) G_q}$. For the first term we have
\begin{eqnarray}\label{espo1}
\int_{B_{r_\eps}} e^{-4\pi \alpha G_p} e^{\ph_\eps} dv_g &=&\eps \int_{B_{r_\eps}} e^{2\alpha \log r -4\pi \alpha A(p) -4\pi \alpha \sigma }  e^{-2\log(\eps+r^{2(1+\alpha)})} dv_g =\nonumber \\
&=&\eps e^{-4\pi \alpha A(p)} \int_{B_{r_\eps}} \frac{r^{2\alpha}}{(\eps + r^{2(1+\alpha)})^2}(1+o_\eps(1)) dv_g
=\nonumber \\
&=&\frac{\pi e^{-4\pi \alpha A(p)}}{ 1+\alpha} \frac{\gamma_\eps^{2(1+\alpha)}}{1+\gamma_\eps^{2(1+\alpha)}}(1 +o_\eps(1)) = \nonumber \\ &=&\frac{\pi e^{-4\pi \alpha A(p)}}{ 1+\alpha} +o_\eps(1).
\end{eqnarray}
Since $\tilde{h}$ is smooth in a neighbourhood of $p$ we obtain
\begin{equation}\label{espo2}
\int_{B_{r_\eps}}\left(\tilde{h}-\tilde{h}(p)\right)  e^{-4\pi\alpha G_p} e^{\ph_\eps} dv_g = o_\eps(1) \int_{B_{r_\eps}}  e^{-4\pi\alpha G_p} e^{\ph_\eps} dv_g = o_\eps(1)
\end{equation} and
\begin{eqnarray}\label{espo3}
\left|\int_{B_\delta\bs B_{r_\eps}} h e^{\ph_\eps} dv_g \right|&=& \left|\int_{B_\delta\bs B_{r_\eps}}\tilde{h} e^{-4\pi \alpha G_p }e^{\ph_\eps} dv_g \right| \le \sup_{B_\delta} |\tilde{h}|\int_{B_\delta\bs B_{r_\eps}}  e^{-4\pi \alpha G_p }e^{\ph_\eps}dv_g=\nonumber\\
&=& \eps e^{C_\eps} \sup_{B_\delta}|\tilde{h}|\int_{B_\delta \bs B_{r_\eps}} e^{4\pi(2+\alpha) G_p } e^{-\ov{\rho} \eta_\eps \sigma } dv_g =  \nonumber\\ 
&=&O(\eps)\int_{B_\delta \bs B_{r_\eps}} e^{4\pi(2+\alpha) G_p } dx  
=  O(\eps) \int_{B_\delta \bs B_{r_\eps}} \frac{1}{|x|^{2(2+\alpha)}}   dx=\nonumber \\
&=& O(\eps)\left( \frac{1}{r_\eps^{2(1+\alpha)}}-\frac{1}{\delta^{2(1+\alpha)}}\right) = O\left(\frac{1}{\gamma_\eps^{2(1+\alpha)}}\right) +O(\eps)=o_\eps(1).
\end{eqnarray}
Finally
\begin{equation}\label{espo4}
\int_{\Sigma\bs B_\delta} h e^{\ph_\eps} dv_g = \eps e^{C_\eps} \int_{\Sigma\bs B_\delta} h e^{\ov{\rho} G_p} dv_g = O(\eps)
\end{equation}
so by (\ref{espo1}), (\ref{espo2}), (\ref{espo3}) and (\ref{espo4}) we have
\begin{equation}\label{p3}
\int_{\Sigma} h  e^{\ph_\eps} dv_g  = \frac{\pi \tilde{h}(p) e^{-4\pi \alpha A(p)}}{ 1+\alpha} +o_\eps(1).
\end{equation}

Using \eqref{grad}, \eqref{media} and \eqref{p3} we get
$$
\lim_{\eps \to 0}J(\ph_\eps)
=-\ov{\rho}\left( 1+4\pi A(p)+ \log\left( \frac{1}{|\Sigma|} \frac{\pi \tilde{h}(p) }{ 1+\alpha} \right) \right) =$$
$$
=-\ov{\rho}\left( 1+\log\frac{\pi}{|\Sigma|}+\max_{\xi\in \Sigma ,\beta(\xi)=\alpha}  \left\{4\pi A(\xi)+ \log  \left(\frac{K(\xi)}{1+\alpha}\prod_{q\in S, q\neq \xi} e^{-4\pi\beta(q) G_q(\xi)}\right)\right\}\right).
$$
This, together with proposition \ref{prop1}, completes the proof of theorem \ref{teo 1}.

\section{Onofri's inequalities on $S^2$}
In this section we will consider the special case of the standard sphere $(S^2,g_0)$ with $m\le 2$ and $K\equiv 1$. We fix  $\alpha_1,\alpha_2\in \R$ with $-1<\alpha_1\le\alpha_2$ and as before we consider  the singular weight
\begin{equation*}
h= e^{-4\pi \alpha_1 G_{p_1}-4\pi \alpha_2 G_{p_2}}. 
\end{equation*}
In order to apply theorem \ref{teo 1} and obtain sharp versions of \eqref{troy}, we need to study the existence of minimum points for the functional $J$.
Let us fix a system of coordinates $(x_1,x_2,x_3)$ on $\R^3$ such that $p_1=(0,0,1)$. If $\min\{\alpha_1,\alpha_2\}\ge0$, $h$ is smooth in $S^2$ and the Kazdan-Warner identity (see \cite{KW}) states that any solution of \eqref{eq nonsing} has to satisfy
$$
\int_{S^2} \nabla h \cdot \nabla x_i \; e^u\; dv_{g_0}= \left(2-\frac{\rho}{4\pi}\right) \int_{S^2} h e^{u} x_i \;dv_{g_0}\quad i=1,2,3.
$$
We claim that if $p_2=-p_1$ the same identity holds, at least in the $x_3$-direction, even when $h$ is singular.

\begin{lemma}\label{lemma traccia}
Let $u$ be a solution of \eqref{eq nonsing} on $S^2$, then there exist $C,\delta_0>0$ such that\\[0.3cm]
\begin{tabular}{ll}
 \tabitem $|\nabla u(x)|\le C d(x,p_i)^{2\alpha_i +1}$  & if $\alpha_i<- \frac{1}{2}$;  \\[0.3cm]
 \tabitem $|\nabla u(x)|\le C  \left(-\log d(x,p_i)\right)\quad $  & if $\alpha_i=- \frac{1}{2}$; \\[0.3cm]
 \tabitem  $|\nabla u(x)|\le C$  & if $\alpha_i>- \frac{1}{2}$;\\[0.3cm]
\end{tabular}

for $0<d(x,p_i)<\delta_0$, $\;i=1,2.$
\end{lemma}
\begin{proof} Let us fix $0<r_0<\frac{1}{2}\min\{\frac{\pi}{2},d(p_1,p_2)\}$ and $i\in \{1,2\}$.
If $\alpha_i>- \frac{1}{2}$ then, by standard elliptic regularity, $u \in C^1(\ov{B_{r_0}(p_i)})$ and the conclusion holds for $\delta_0=r_0$ and $C=\|\nabla u\|_{L^\infty(B_{r_0}(p_i))}$. Let us now assume $\alpha_i \le-\frac{1}{2}$. We know that $h(y)\le C_1 d(y,p_i)^{2\alpha_i}$ for $y\in B_{2r_0}(p_i)$ so, if $\delta_0 < r_0$, by Green's representation formula we have
$$
|\nabla u|(x)\ \le \rho e^{\|u \|_\infty} \int_{S^2} \frac{h(y)}{d(x,y)} dv_{g_0}(y)\le \frac{\rho e^{\|u\|_\infty} \|h\|_{L^1(S^2)}}{r_0} +  \rho e^{\|u \|_\infty} C_1 \int_{B_{r_0}(x)} \frac{d(y,p_i)^{2\alpha_i}}{d(x,y)}dv_{g_0}(y)
.$$
Let $\pi$ be the stereographic projection from the point $-p_i$. It is easy to check that there exist $C_2,C_3>0$ such that
$$
C_2 \; d(q,q') \le |\pi(q)-\pi(q')|\le C_3\; d(q,q')
$$
$\forall \; q,q'\in B_{\frac{\pi}{2}}(p_i)$. 
Thus we have
$$
\int_{B_{r_0}(x)} \frac{d(y,p_i)^{2\alpha_i}}{d(x,y)}dv_{g_0}(y) \le \int_{B_{\frac{\pi}{2}}(p_i)} \frac{d(y,p_i)^{2\alpha_i}}{d(x,y)}dv_{g_0}(y) \le C_4 \int_{\{|z|\le 1\}} \frac{|z|^{2\alpha_i}}{|\pi(x)-z|} dz =$$
$$
= C_4 |\pi (x)|^{2\alpha_i+1}\int_{\left\{|z|\le \frac{1}{|\pi(x)|}\right\}} \frac{|z|^{2\alpha_i}}{\left|\frac{\pi(x)}{|\pi(x)|}-z\right|}dz \le  C_5 d(x,p_i)^{2\alpha_i+1} \int_{\left\{|z|\le \frac{1}{|\pi(x)|}\right\}} \frac{|z|^{2\alpha_i}}{\left|\frac{\pi(x)}{|\pi(x)|}-z\right|} dz.
$$
Notice that 
$$
\int_{\left\{|z|\le \frac{1}{|\pi(x)|}\right\}} \frac{|z|^{2\alpha_i}}{\left|\frac{\pi(x)}{|\pi(x)|}-z\right|} dz \le  \frac{1}{2^{2\alpha_i}} \int_{\left\{\left|\frac{\pi(x)}{|\pi(x)|}-z\right|\le \frac{1}{2}\right\}}\frac{1}{\left|\frac{\pi(x)}{|\pi(x)|}-z\right|} dz +2 \int_{\{|z|\le 2\}} |z|^{2\alpha_i}dz + 2\int_{\left\{2\le |z|\le \frac{1}{|\pi(x)|}\right\}} |z|^{2\alpha_i-1} dz \le
$$
$$
\le C_6+ 2  \int_{\left\{2\le |z|\le \frac{1}{|\pi(x)|}\right\}} |z|^{2\alpha_i-1} dz.
$$
If $\alpha_i<-\frac{1}{2}$
$$
 \int_{\left\{2\le |z|\le \frac{1}{|\pi(x)|}\right\}} |z|^{2\alpha_i-1} dz \le C_7,
$$
while if $\alpha_i=-\frac{1}{2}$
$$
\int_{\left\{2\le |z|\le \frac{1}{|\pi(x)|}\right\}} |z|^{2\alpha_i-1} dz = 2\pi \log\left( \frac{1}{2|\pi(x)|}\right) \le C_8 \left( - \log d(x,p_i) \right).
$$
Thus we get the conclusion for $\delta_0$ sufficiently small. 
\end{proof}

In any case there exists $s\in [0,1)$ such that
\begin{equation}\label{stima gradiente}
|\nabla u(x)|\le C d(x,p_i)^{-s} \left(-\log d(x,p_i)\right)
\end{equation}
for $0<d(x,p_i)<\delta_0$, $\;i=1,2.$

\begin{prop}\label{prop poho}
If $p_2=-p_1$ then any solution of \eqref{eq nonsing} satisfies
$$
\int_{S^2} \nabla h \cdot \nabla x_3 \; e^u \; dv_{g_0} = \left(2-\frac{\rho}{4\pi}\right) \int_{S^2} h e^{u} x_3 \; dv_{g_0}.
$$
\end{prop}
\begin{proof}
Without loss of generality we may assume
\begin{equation}\label{norm sfera}
\int_{S^2} h e^{u} dv_{g_0}=1.
\end{equation}
Let us denote $S_\delta= S^2\bs B_{\delta}(p_1)\cup B_\delta (p_2)$.  Since $u$ is smooth in $S_\delta$, multiplying \eqref{eq nonsing} by $\nabla u\cdot \nabla x_3$ and integrating on $S_\delta$ we have 
\begin{equation}\label{sfera1}
-\int_{S_\delta} \Delta u  \; \nabla u\cdot \nabla x_3 \:  dv_{g_0}=\rho  \int_{S_\delta} \left(h\; e^{u} -\frac{1}{4\pi}\right) \nabla u \cdot \nabla x_3 \; dv_{g_0}
\end{equation}
Integrating by parts we obtain 
$$
-\int_{S_\delta} \Delta u  \; \nabla u\cdot \nabla x_3\; dv_{g_0} =  \int_{S_\delta} \nabla u \cdot \nabla (\nabla u \cdot \nabla x_3)dv_{g_0} +\sum_{i=1}^2 \int_{\partial B_\delta(p_i)}\nabla u \cdot \nabla x_3 \D{u}{n} d\sigma_{g_0} 
$$
and by \eqref{stima gradiente}
$$
\left|\int_{\partial B_\delta(p_i)}\nabla u \cdot \nabla x_3 \;  \D{u}{n}\; d\sigma_{g_0}\right|\le \int_{\partial B_{\delta}(p_i)} |\nabla u|^2 |\nabla x_3| d\sigma_{g_0} =    O(\delta^{2(1-s)} \log^2 \delta)=o_\delta(1).
$$
Using the  identities
$$
\nabla u \cdot \nabla (\nabla u \cdot \nabla x_3)= \frac{1}{2}\nabla |\nabla u|^2 \cdot \nabla x_3 -x_3|\nabla u|^2 
$$
and $$-\Delta x_3 = 2x_3,$$
and applying again \eqref{stima gradiente} to estimate the boundary term, we get
$$
-\int_{S_\delta} \Delta u  \; \nabla u\cdot \nabla x_3 \; dv_{g_0}=\int_{S_\delta}  \frac{1}{2}\nabla |\nabla u|^2 \cdot \nabla x_3 \; dv_{g_0} -\int_{S_\delta} x_3 |\nabla u|^2  dv_{g_0}+o_\delta(1)= $$
$$
=-\frac{1}{2}\int_{S_\delta} \Delta x_3 \; |\nabla u|^2 dv_{g_0} -\sum_{i=1}^2\int_{\partial B_\delta(p_i)} |\nabla u|^2\D{x_3}{n} d\sigma_{g_0} -\int_{S_\delta} x_3 |\nabla u|^2 dv_{g_0} =o_\delta(1).$$

Thus (\ref{sfera1}) becomes
\begin{equation}\label{sfera2}
\int_{S_\delta} h e^{u}\nabla u \cdot \nabla x_3 \; dv_{g_0} - \frac{1}{4\pi} \int_{S_\delta} \nabla u \cdot \nabla x_3\; dv_{g_0}=o_\delta(1).
\end{equation}

Moreover
$$
\int_{S_\delta} \nabla u \cdot \nabla x_3\; dv_{g_0}= -\int_{S_\delta} \Delta u \; x_3\; dv_{g_0} -\sum_{i=1}^2 \int_{\partial B_\delta(p_i) } x_3 \D{u}{n} \; d\sigma_{g_0}= $$
\begin{equation*}
=\rho \int_{S_\delta} \left(h e^u-\frac{1}{4\pi}\right)x_3 \;dv_{g_0}+ O(\delta^{1-s}(-\log \delta))
=\rho \int_{S_\delta} h e^u x_3 \;dv_{g_0}+o_\delta(1) 
\end{equation*}
and 
$$
\int_{S_\delta } h e^{u} \; \nabla u \cdot \nabla x_3\; dv_{g_0} = \int_{S_\delta}  \nabla e^{u} \cdot h \nabla x_3 \; dv_{g_0} = - \int_{S_\delta} e^u \dv(\;h \nabla x_3) dv_{g_0} -\sum_{i=1}^2\int_{\partial B_\delta(p_i)} h e^{u} \D{x_3}{n}\; d\sigma_{g_0}=$$
$$
=-\int_{S_\delta}  \nabla h \cdot \nabla x_3 \; e^{u} \; dv_{g_0} + 2 \int_{S_\delta} h e^u x_3 dv_{g_0} +O(\delta^{2(1+\alpha)}).
$$
Thus  by \eqref{sfera2} we have
$$
\int_{S_\delta}  \nabla h \cdot \nabla x_3 \; e^{u}\; dv_{g_0}  = \left(2-\frac{\rho}{4\pi}\right) \int_{S_\delta} h e^u x_3\; dv_{g_0}+o_\delta(1).
$$ 
Since $u$ is continuous on $S^2$ and $h, \nabla h \cdot \nabla x_3\in L^1(S^2)$ as $\delta\to 0 $ we get the conclusion.
\end{proof}

\begin{oss}
In this proof there is no need to assume $K\equiv 1$.
\end{oss}

Assuming $p_1=(0,0,1)$ and $p_2=(0,0,-1)$, one may easily  verify that
$$
G_{p_1}(x)=-\frac{1}{4\pi} \log(1-x_3) -\frac{1}{4\pi} \log \left(\frac{e}{2}\right)
$$
and 
$$
G_{p_2}(x)=-\frac{1}{4\pi} \log(1+x_3) -\frac{1}{4\pi} \log \left(\frac{e}{2}\right),
$$
so that 
$$
\nabla h \cdot\nabla x_3 = -4\pi  h (\alpha_1 \nabla G_1+ \alpha_2 \nabla G_2)\cdot \nabla x_3=  (\alpha_2 -\alpha_1)h - (\alpha_1+\alpha_2)h x_3.
$$
Thus we can rewrite the identity in proposition \ref{prop poho} as
\begin{equation}\label{poho}
\alpha_2-\alpha_1 = \left(2-\frac{\rho}{4\pi} +\alpha_1+\alpha_2\right) \int_{S^2} h e^u x_3 \;dv_{g_0}.
\end{equation}
\begin{proof}[Proof of theorem \ref{una sing}.]
Assume $m=1$ (i.e. $\alpha_2=0$). We claim that equation \eqref{eq nonsing} has no solutions for $\rho = \ov{\rho} = 8\pi(1+\min\{0,\alpha_1\})$, unless $\alpha_1 =0$. Indeed if $u$ were a solution of   \eqref{eq nonsing} satisfying \eqref{norm sfera}, then applying \eqref{poho} with $\rho = \ov{\rho}$ we would get
$$
-\alpha_1 = \left( \alpha_1-2\min\{0,\alpha_1\}\right) \int_{S^2} h e^u x_3\;dv_{g_0}
$$
so that, if $\alpha_1\neq 0$,
$$
\left|\int_{S^2} h e^{u} x_3\; dv_{g_0} \right| =1.
$$
This contradicts \eqref{eq nonsing}. In particular we proved non-existence of minimum points for $J_{\ov{\rho}}$ so we can exploit theorem \ref{teo 1} and \eqref{C e inf} to prove that $\eqref{troy}$ holds with
$$
C=  \max_{p\in S^2 ,\beta(p)=\alpha } \left\{\log  \left(\frac{1}{1+\alpha}\prod_{q\in S, q\neq p} e^{-4\pi\beta(q) G_q(p)}\right)\right\}.
$$
If $\alpha_1<0$ one has $$C=-\log(1+\alpha_1).$$
If $\alpha_1>0$,
$$
C= \max_{p\in S^2\bs\{p_1\} } \left\{ -4\pi \alpha_1  G_{p_1}(p)\right\}=-4\pi \alpha_1 G_{p_1}(p_2) = \alpha_1.
$$
\end{proof}

\begin{proof}[Proof of theorem \ref{due sing}]
As in the previous proof, applying  \eqref{poho} with $\rho = \ov{\rho} =8\pi(1+\alpha_1)$, we obtain that any critical point of \eqref{eq nonsing} for which \eqref{norm sfera} holds has to satisfy
$$
\alpha_2-\alpha_1= (\alpha_2 -\alpha_1)\int_{S^2} h e^u x_3 dv_{g_0}.
$$ 
Since $\alpha_1 \neq \alpha_2$ one has 
$$
\int_{S^2} h e^u x_3 dv_{g_0}=1
$$
that is impossible. Thus $J_{\ov{\rho}}$ has no critical point and by theorem \ref{teo 1} one has 
$$C=\log  \left(\frac{1}{1+\alpha_1}e^{-4\pi \alpha_2 G_{p_2}(p_1)}\right)= \alpha_2-\log(1+\alpha_1).
$$
\end{proof}

Now we assume $\alpha_1=\alpha_2<0$. In this case  identity (\ref{poho}) gives no useful condition.   Let us denote by $\pi$ the stereographic projection  from the point $p_1$. It is easy to verify that  $u$ satisfies  \eqref{eq nonsing}  and \eqref{norm sfera} if and only if   
$$
v:= u\circ \pi^{-1} +(1+\alpha)\log\left(\frac{4}{(1+|y|^2)^2}\right)+2\alpha \log\left(\frac{e}{2}\right)
$$
solves 
\begin{equation}\label{equazione piano}
-\Delta_{\R^2} v = 8\pi (1+\alpha)|y|^{2\alpha} e^{v}  
\end{equation}
in $\R^2$ and
$$
\int_{\R^2} |y|^{2\alpha} e^{v} dy =1.
$$
As we pointed out in the proof of lemma \ref{lemma risc}, equation \eqref{equazione piano} has a one-parameter family of solutions:
$$
v_\lambda(y)= -2 \log \left(1+\frac{\pi}{1+\alpha} e^l |y|^{2(1+\alpha)}\right)
$$
$l\in \R$.  Thus we have a corresponding family $\{u_{\lambda,c}\}$ of critical points of $J_{\ov{\rho}}$ given by the expression
\begin{equation}\label{mini}
u_{\lambda,c}\circ \pi^{-1}(y) = 2\log\left(\frac{(1+|y|^2)^{1+\alpha}}{1+ \lambda |y|^{2(1+\alpha)}}\right)+c,
\end{equation}
$c\in \R,\lambda>0$. A priori we do not know whether these critical points are minima for $J_{\ov{\rho}}$ (as it happens for $\alpha =0$), so a direct application of \ref{teo 1} is not possible. However, we can still get the conclusion by comparing $J_{\ov{\rho}}(u_{\lambda,c})$ with the blow-up value provided by theorem \ref{teo 1}.

\begin{proof}[Proof of theorem \ref{due sing uguali}]
Let us first compute $J(u_{\lambda,c})$. Let $\ph_t:S^2\ra S^2$ be the conformal transformation defined by $\pi(\ph_t (\pi^{-1}(y)))= t y$. It is not difficult to prove that $\forall \; t>0$
$$
J_{\ov{\rho}}(u)=J_{\ov{\rho}}(u \circ \ph_t + (1+\alpha)\log|\det d\ph_t|);
$$
in particular, since $$u_{\lambda,c}= u_{1,0}\circ \ph_{\lambda^\frac{1}{2(1+\alpha)}} +(1+\alpha) \log|\det\ph_{\lambda^\frac{1}{2(1+\alpha)}}| +c-\log\lambda,$$ 
we have that $J(u_{\lambda,c})$ does not depend on $\lambda$ and $c$. Thus we may assume $\lambda=1$ and $c=0$. A simple computation shows that
\begin{equation}\label{fine1}
\int_{S^2} h \; e^{u_{1,0}} dv_{g_0}= 4 e^{2\alpha} \int_{\R^2}\frac{|y|^{2\alpha}}{\left(1+|y|^{2(1+\alpha)}\right)^2} dy = \frac{4 e^{2\alpha} \pi}{1+\alpha}.
\end{equation}

Since $u_{1,0}(p_1)=0$ and $u_{1,0}$ solves
\begin{equation*}
-\Delta  u_{1,0} = \omega \;  h \;  e^{u_{1,0}} -2(1+\alpha) \quad \mbox{with} \quad\omega:=2(1+\alpha)^2 e^{-2\alpha}
\end{equation*}
one has
$$
\int_{S^2} u_{1,0} \; dv_{g_0} = 4\pi  \int_{S^2} \Delta u_{1,0}\; G_{p_1} dv_{g_0}= -4\pi \omega\int_{S^2} h e^{u_{1,0}} G_{p_1} dv_{g_0}
$$
and
$$
\frac{1}{2}\int_{S^2}|\nabla u_{1,0}|^2dv_{g_0} +2(1+\alpha) \int_{S^2} u_{1,0}\; dv_{g_0} =  \frac{1}{2} \omega \int_{S^2} h e^{u_{1,0}} u_{1,0} \;  dv_{g_0} + (1+\alpha)\int_{S^2} u_{1,0}\; dv_{g_0} =
$$
\begin{equation}\label{fine2}
= \frac{\omega}{2}\int_{S^2} h e^{u_{1,0}} ( u_{1,0}-\ov{\rho} G_{p_1}) dv_{g_0}.
\end{equation}
Since 
$$
G_{p_1}(\pi^{-1}(y)):= \frac{1}{4\pi} \log(1+|y|^2) -\frac{1}{4\pi}
$$
we get 
$$
\int_{S^2} h e^{u_{1,0}} ( u_{1,0}-\ov{\rho} G_{p_1}) = 2(1+\alpha)\int_{S^2} h e^{u_{1,0}} dv_{g_0}-8e^{2\alpha}\int_{\R^2} \frac{|y|^{2\alpha} \log\left({1+|y|^{2(1+\alpha)}}\right)}{\left(1+|y|^{2(1+\alpha)}\right)^2} dy= 
$$
\begin{equation}\label{fine3}
= 8\pi e^{2\alpha}-\frac{8 \pi e^{2\alpha}}{1+\alpha} \int_{0}^{+\infty } \frac{\log(1+s)}{(1+s)^2}ds = \frac{8\pi \alpha e^{2\alpha}}{1+\alpha}.
\end{equation}

Using \eqref{fine1}, \eqref{fine2} and \eqref{fine3} we obtain
$$
J(u_{\lambda,c})=J(u_{1,0})= 8\pi (1+\alpha)\left(\log(1+\alpha)-\alpha\right) \qquad \forall\; \lambda>0, c \in \R.
$$
To conclude the proof it is sufficient to observe that $u_{\lambda,c}$ have to be minimum points for $J_{\ov{\rho}}$ that is 
$$
\inf_{H^1(S^2)} J_{\ov{\rho}} = 8\pi (1+\alpha)\left(\log(1+\alpha)-\alpha\right).
$$
Indeed if this were false then $J_{\ov{\rho}}$ would have no minimum points but, by theorem \ref{teo 1}, we would get 
$$
\inf_{H^1(S^2)} J_{\ov{\rho}} = 8\pi (1+\alpha)\left(\log(1+\alpha)-\alpha\right)= J(u_{\lambda,c}).
$$ This is clearly a contradiction.
\end{proof}

\begin{oss}
There is no need to assume $p_1=-p_2$.
\end{oss} 
Indeed given two arbitrary points $p_1,p_2\in S^2$ with $p_1\neq p_2$ it is always possible to find a conformal diffeomorphism $\ph:S^2 \ra S^2$ such that $\ph^{-1}(p_1) = -\ph^{-1}(p_2)$. Moreover   one has
$$
J_{\ov{\rho}}(u)= \widetilde{J_{\ov{\rho}}}(u\circ \ph  + (1+\alpha)\log|\det d \ph|)+c_{\alpha,p_1,p_2}
$$
$\forall \; u \in H^1(S^2)$, where $\widetilde{J}$ is the Moser-Trudinger functional associated to 
$$
\widetilde{h}= e^{-4\pi \alpha G_{\ph^{-1}(p_1)}-4\pi \alpha G_{\ph^{-1}(p_2)}}.
$$
and $c_{\alpha,p_1,p_2}$ is an explicitly known constant depending only on $\alpha$, $p_1$ and $p_2$.
In particular one can still compute $\min_{H^1(S^2)} J_{\ov{\rho}}$ and describe the minimum points of $J_{\ov{\rho}}$ in terms of $\ph$ and the family \eqref{mini}.

\section*{Acknowledgements}
The author would like to express his gratitude to Professor Andrea Malchiodi for many valuable discussions and for his guidance during the preparation of this work.

\bibliographystyle{plain}
\bibliography{biblio}

\end{document}